\documentclass[11pt, letterpaper]{article}
\usepackage{amsmath, amsthm, amssymb, amsfonts} 
\usepackage{mathrsfs} 
\usepackage{bm} 
\usepackage{graphicx} 
\usepackage{enumerate} 
\usepackage{geometry} 
\usepackage{setspace} 
\usepackage{lmodern} 
\usepackage{hyperref} 
\usepackage{color} 
\usepackage{xcolor} 
\usepackage{url} 
\usepackage{mathtools}
\usepackage{enumitem}
\usepackage{changepage}
\usepackage{microtype}
\usepackage{authblk}
\usepackage{amsfonts}
\usepackage{mathrsfs,amscd,amssymb,amsthm,amsmath,bm,graphicx,psfrag,subfigure,url,mathtools}
\usepackage{pict2e}
\usepackage{psfrag,amsmath}
\usepackage{tikz}
\usepackage{indentfirst}
\usepackage{hyperref}
\usepackage{bookmark}
\usepackage{enumerate}
\usepackage{latexsym,euscript,epic,eepic,color}
\usepackage{multirow}
\usepackage{multicol}
\usepackage{longtable}
\usepackage{adjustbox}

\usepackage{setspace}
\usepackage{epstopdf}
\allowdisplaybreaks
\usepackage{authblk}
\usepackage{pifont}

\theoremstyle{plain}
\newtheorem{theorem}{Theorem}[section]          
\newtheorem{lemma}{Lemma}[section]              
\newtheorem{claim}{Claim} 
\newtheorem{problem}{Problem}

\theoremstyle{definition}
\newtheorem{definition}{Definition}

\newcommand{\ar}{\operatorname{ar}}
\newcommand{\ex}{\operatorname{ex}}
\newcommand{\calR}{\mathcal{R}}
\newcommand{\calH}{\mathcal{H}}

\hypersetup{
    colorlinks=true,
    linkcolor=blue,
    citecolor=red,
    urlcolor=magenta,
}

\newcommand{\keywords}[1]{%
  \par\vspace{6pt}\noindent\textbf{Keywords: }#1\par
}

\newcommand{\MSC}[2][2020]{%
  \par\vspace{3pt}\noindent\textbf{MSC(#1): }#2\par
}
\title{Anti-Ramsey Number for Suspension of Edge-Critical Graphs}
\vspace{6mm}

\author{Shuchao Li } 

\author{Haojie Zheng\thanks{Corresponding author. \\
\hspace*{2em}E-mail address: lscmath@ccnu.edu.cn (S. Li), zhj9536@126.com (H. Zheng)}}

\affil{School of Mathematics and Statistics, and Hubei Key Lab--Math. Sci.,\linebreak Central China Normal University, Wuhan 430079, China}

\date{\today}

\allowdisplaybreaks
\begin{document}
\baselineskip=0.23in

\maketitle

\begin{abstract}
An edge-colored graph is called a rainbow graph if all its edges have distinct colors. The \textit{anti-Ramsey number}, denoted by $\ar(n,F),$ for a fixed graph $F$ and a positive integer $n$, is the maximum number of colors used in an edge-coloring of the complete graph $K_n$ that contains no rainbow copy of $F$. Meanwhile, the \textit{Tur\'an number}, denoted by $\ex(n,F),$ for graph $F$ and $n$, is the maximum number of edges in an $n$-vertex graph that does not contain $F$ as a subgraph. For a vertex $v$ and a multiset $\mathcal{H}$ of graphs, the \textit{suspension} $\mathcal{H} + v$ of $\mathcal{H}$ is the graph obtained by connecting the vertex $v$ to all vertices of $H$ for each $H \in \mathcal{H}$. Let integers $k\ge 1$ and $r\ge 2$ be fixed, and suppose that $\mathcal{H}_{k+1}=\{H_1, H_2, \ldots, H_{k+1}\}+v$ satisfying $H_1, H_2, \ldots, H_{k+1}$ are pairwise vertex-disjoint edge-critical graphs, and $\chi(H_i)=r$ for $i=1,2,\ldots, k+1$. 
In this paper, we determine 
$
\ar(n,\mathcal{H}_{k+1}) 
$
for $k\ge 1$, $r\ge 2$ and sufficiently large $n$. This result unifies and generalizes a result of Liu et al. (arXiv:2411.08475) concerning the friendship graph, and a result of Lu et al. (arXiv:2507.13165) on the intersecting cliques.
\end{abstract}
\MSC{05C15; 05C35}

\section{Introduction}

All graphs considered in this paper are finite and simple. For a graph $G$, let $V(G)$, $E(G)$, $e(G)$, $\delta(G)$, $\Delta(G)$, $\chi(G)$, and $\nu(G)$ denote its vertex set, edge set, number of edges, minimum degree, maximum degree, chromatic number, and matching number, respectively. 
For $X\subseteq V(G)$ and $x\in V(G)$, write $N_X(x)=N_G(x)\cap X$ and $d_X(x)=|N_X(x)|$.
For a positive integer $k$, let $[k]=\{1,\ldots,k\}$. The symbol $\sim$ stands for adjacency between the two vertices in question.

For a vertex $v$ and a multiset $\mathcal{H}$ of graphs, the \textit{suspension} $\mathcal{H} + v$ of $\mathcal{H}$ is the graph obtained by connecting the vertex $v$ to all vertices of $H$ for each $H \in \mathcal{H}$, in which $v$ is called the \textit{center} and each subgraph $H+v$ is called 
a \emph{branch}. A graph is said to be properly colored if each vertex is colored so that adjacent vertices have distinct colors. If $H$ can be properly colored by $k$ colors, then we say $H$ is $k$-\textit{colorable}. The \textit{chromatic number} $\chi(H)$ is $k$ if $H$ is $k$-colorable and not $(k-1)$-colorable. We say that $e \in E(H)$ is a \textit{color}-\textit{critical edge} of $H$ if $\chi(H - e) < \chi(H)$. A graph is \textit{edge-critical} if it contains an edge whose deletion reduces its chromatic number. Clearly, every odd cycle and every non-trivial complete graph is edge-critical. In the sequel, a color-critical edge is simply called a critical edge.

Let integers $k\ge 1$ and $r\ge 2$ be fixed, and suppose that $\{H_1, H_2, \ldots, H_{k+1}\}$ is a collection of $k+1$ pairwise vertex-disjoint edge-critical graphs each with chromatic number $r$. Here, $H_1, H_2, \ldots, H_{k+1}$ are not required to be isomorphic.
Let $\mathcal{H}_{k+1}=\{H_1,\ldots,H_{k+1}\}+v$ be the suspension of $\{H_1, H_2, \ldots, H_{k+1}\}$. Then, for every $j\in[k+1]$, define the $k$-branch sub-suspension $\mathcal{H}_k^j =\{H_i:i\in[k+1]\setminus\{j\}\}+v.$

The \textit{Tur\'an number}, denoted by $\ex(n,F),$ for a fixed graph $F$ and positive integer $n$, is the maximum number of edges in an $n$-vertex graph that does not contain $F$ as a subgraph and the Tur\'an graph serves a pivotal role. 
Given integers $n,r$ with $n\geqslant r\geqslant2$, the Turán graph $T_r(n)$ is an $n$-vertex complete $r$-partite graph whose parts differ in size by at most one (each part has size $\lceil n/r\rceil$ or $\lfloor n/r\rfloor$). We write $t_r(n)$ for the edge number of $T_r(n)$. As stated by the classic Tur\'an’s theorem \cite{c1t}:
$
\operatorname{ex}(n, K_{r+1}) = t_r(n) = \left(1 - \frac{1}{r} + o(1)\right)\binom{n}{2},
$
and $T_r(n)$ is the sole extremal graph for this Turán-type extremal problem.

Determining $\operatorname{ex}(n,F)$ is one of the most important problems in extremal graph
theory. Although the Tur\'an numbers of non-bipartite graphs are asymptotically determined
by Erd\H{o}s-Stone-Simonovits theorem~\cite{6,7}, it is still a challenge to determine the exact values of Tur\'an
number for many non-bipartite graphs. The Tur\'an numbers are known exactly for only a few specific graphs (e.g. see \cite{CLL,8,10,13,Simonovits,19,20,21,22}). Among all the existing results, the Tur\'an number of the graph consisting of some specific graphs that intersect in exactly one common vertex is widely studied (e.g. see \cite{CGPW,EFGG,HLZ,11,14,18}).

An edge-colored graph is called a \textit{rainbow graph} if all its edges have distinct colors. The \textit{anti-Ramsey number}, denoted by $\ar(n,F),$ for a fixed graph $F$ and a positive integer $n$, is the maximum number of colors used in an edge-coloring of the complete graph $K_n$ that contains no rainbow copy of $F$. The anti-Ramsey theory was initiated by Erd\H{o}s, Simonovits and S\'os \cite{ESS}, who revealed a close connection among anti-Ramsey problems, Tur\'an-type extremal problems, and the chromatic properties of the forbidden graph. Since then, a great many results have been established for a broad range of graphs embedded within complete graphs, with matchings receiving particularly thorough attention (see \cite{c3,c6,c7,c10,c12,Schiermeyer}).
In recent years, subsequent work extended the family of host graphs far beyond complete graphs: alongside complete graphs, the anti-Ramsey numbers of matchings have also been extensively explored for planar graphs \cite{c2,c13,c14,c22,c31,c32}, bipartite graphs \cite{c23,c26,c27}, and hypergraphs \cite{c9,c18,c29,c35}. For further interesting topics concerning anti-Ramsey theory, readers may refer to \cite{FMO,JP}.


The present work concerns the anti-Ramsey number for the suspension of edge-critical graphs with the same chromatic number. This provides a unified framework for several graph families formed by subgraphs sharing a common vertex. Two typical examples are the friendship graph and the intersecting cliques.  Liu, Lu and Luo \cite{LiuLuLuo} determined the anti-Ramsey number of friendship graph and Lu, Luo and Ma \cite{LuLuoMa} determined the anti-Ramsey number of intersecting cliques.

Motivated by \cite{LiuLuLuo,LuLuoMa}, we determine the anti-Ramsey number for a suspension of $k+1$ edge-critical graphs.
More importantly, our approach reveals why the edge-critical condition arises naturally in this problem. The essential property is not the particular structure of these edge-critical graphs, but the existence, in each such graph, of a critical edge whose deletion reduces its
chromatic number. Such a critical edge provides the local structure needed to embed the suspension of the corresponding edge-critical graph into the host graph. From this perspective, friendship graphs and intersecting cliques are two special cases of critical-edge structure. This observation allows us to treat suspensions formed by arbitrary edge-critical graphs with the same chromatic number in a unified framework. 
Our main result can be described as follows.
\begin{theorem}\label{thm:main}
Let $k\geq 1$ and $r\geq 2$ be fixed integers, and let
$H_1,\ldots,H_{k+1}$ be edge-critical graphs with
$\chi(H_i)=r$ for every $i\in[k+1]$. There is a function $\phi\bigl(r, |V(H_1)|,\ldots,|V(H_{k+1})|\bigr)$
such that if $n\ge\phi\bigl(r, |V(H_1)|,\ldots,|V(H_{k+1})|\bigr)$,
then for every $j\in[k+1]$,
\begin{equation*}\label{e1}
\operatorname{ar}(n,\mathcal{H}_{k+1})=\ex(n,\calH_{k}^{j})+1=t_r(n) +1+
\begin{cases}
k^2 - k & \text{if } k \text{ is odd},\\[4pt]
k^2 - \dfrac{3}{2}k & \text{if } k \text{ is even}.
\end{cases}
\end{equation*}
\end{theorem}


The remainder of the paper is organized as follows. Section~\ref{s2} collects the extremal and structural results needed in the proof. In Section~\ref{s3}, we establish an embedding lemma and use it to handle the case in which every representing graph has large minimum degree. Section~\ref{s4} completes the proof of Theorem~\ref{thm:main}. We conclude in Section~\ref{s5} with a related problem concerning anti-Ramsey numbers of edge-critical graphs.

\section{Preliminaries}\label{s2}

In this section, we collect the extremal and structural results needed in the proof of Theorem~\ref{thm:main}. 
Recall that Tur\'an graph $T_r(n)$ is a balanced complete $r$-partite graph on $n$ vertices.
We shall repeatedly use the identity
\[
t_r(n)-t_r(n-1)
=
\left\lfloor \frac{(r-1)n}{r}\right\rfloor.
\]
Indeed, $T_r(n)$ can be obtained from $T_r(n-1)$ by adding one vertex to a smallest part. The degree of the new vertex is
$
n-\left\lceil\frac{n}{r}\right\rceil
=\lfloor\frac{(r-1)n}{r}\rfloor .
$

We next recall a function concerning a graph $G$ with bounded matching number $\nu(G)$ and maximum degree $\Delta(G)$. For positive integers $\nu$ and $\Delta$, define
\[
f(\nu,\Delta)
=
\max\bigl\{
e(G):\nu(G)\leq \nu,\ \Delta(G)\leq \Delta
\bigr\}.
\]

\begin{definition}[{$k$-good partition}]
\label{def:good-partition}
Let $k,r\geq2$ be integers. A partition $V(G)=V_1\cup V_2\cup\cdots\cup V_r$ is called \emph{$k$-good} if, for every $i\in[r]$, the following conditions
hold:
\begin{enumerate}[label=\textup{(\roman*)}]
\item
$\Delta(G[V_i])\leq k-1$;

\item
$\sum\limits_{j\in[r]\setminus\{i\}}\nu(G[V_j])\leq k-1$;

\item For every $u\in V_i$,
$d_{V_i}(u)+
\sum\limits_{j\in[r]\setminus\{i\}}
\nu\bigl(G[N_G(u)\cap V_j]\bigr)
\leq k-1.
$
\end{enumerate}
\end{definition}

Let $G$ be a graph with a partition of the vertices into $r$ non-empty parts $V(G)=V_1\cup V_2\cup\cdots\cup V_r$.
Define $G_{\rm cr}=G-(\bigcup_{i=1}^{r}E(G[V_i]))$. For distinct $i,j\in[r]$, let $G[V_i,V_j]$ denote the bipartite
subgraph of $G$ consisting of all edges between $V_i$ and $V_j$.
For $x\in V_i$ and $y\in V_j$ with $i\neq j$, we call $xy$ a \emph{missing cross-edge} if $xy\notin E(G)$.

We now state the preliminary results that will be used throughout the
paper.

\begin{lemma}[\cite{AbbottHansonSauer}]
\label{lem:AHS}
For every integer $k\geq2$,
\[
f(k-1,k-1)
=
\begin{cases}
k^2-k, & \text{if $k$ is odd},\\[1mm]
k^2-\dfrac{3k}{2}, & \text{if $k$ is even}.
\end{cases}
\]
Moreover, when $k$ is odd, an extremal graph is the disjoint union of two copies of $K_k$. When $k$ is even, an extremal graph has $2k-1$ vertices,
$k^2-\frac{3k}{2}$ edges, and maximum degree $k-1$.
\end{lemma}

For the degenerate case $k=1$, we set $f(0,0)=0$.

\begin{lemma}[\cite{ChvatalHanson}]
\label{CH}
For every $\nu\geq1$ and $\Delta\geq1$,
\[
f(\nu,\Delta)
=
\nu\Delta
+
\left\lfloor\frac{\Delta}{2}\right\rfloor
\left\lfloor
\frac{\nu}{\lceil\Delta/2\rceil}
\right\rfloor
\leq \nu\Delta+\nu.
\]
\end{lemma}

\begin{lemma}[\cite{CGPW}]
\label{maximal}
Suppose that $G$ has a $k$-good partition $V(G)=V_1\cup V_2\cup\cdots\cup V_r$.
Let $G'$ be a minimal induced subgraph of $G$ for which $e(G')-\sum_{1\leq i<j\leq r}|V_i'||V_j'|$ is maximum, 
where $V_i'=V(G')\cap V_i$ for every $i\in[r]$.
Then the following properties hold:
\begin{enumerate}[label=\textup{(\roman*)}]
\item
$e(G')-\sum\limits_{1\leq i<j\leq r}|V_i'||V_j'|\leq f(k-1,k-1);$

\item For every $i\in[r]$ and every $x\in V_i'$, we have
$0<d_{G'}(x)-|V(G')\setminus V_i'|
\leq k-1-
\sum\limits_{j\in[r]\setminus\{i\}}
\nu\bigl(G'[V_j']\bigr);$

\item 
If $\nu\bigl(G'[V_i']\bigr)\geq2$
for every $i\in[r]$,
then
$e(G')-\sum\limits_{1\leq i<j\leq r}|V_i'||V_j'|
<
f(k-1,k-1).
$
\end{enumerate}
\end{lemma}

Given a graph $H$, a graph $G$ is called $H$-free if it contains no copy of $H$ as a subgraph.
Let $\mathcal{G}_{n,k,r}$ denote a family of graphs, each of which is obtained from Tur\'an graph $T_r(n)$ by embedding a graph with $2k-1$ vertices, $k^2 - \frac{3k}{2}$ edges with maximum degree $k-1$ in one partite set if $k$ is even and embedding two vertex disjoint copies of $K_k$ in one partite set if $k$ is odd. By convention, we have $\mathcal G_{n,1,r}=\{T_r(n)\}$.

\begin{lemma}[\cite{HLZ}]
\label{thm:HLZ-extremal}
Let $k\geq1$ and $r\geq2$ be integers. For every $i\in[k]$, let $H_i$ be an edge-critical graph with $\chi(H_i)=r$, and let
$\mathcal{H}_k=\{H_1,\ldots,H_k\}+v$.
Then, for sufficiently large $n$,
$\ex(n,\mathcal{H}_k)=t_r(n)+f(k-1,k-1)$.
Moreover, $\mathcal{G}_{n,k,r}$ is the family of extremal graphs for
$\mathcal{H}_k$.
\end{lemma}

\begin{lemma}[\cite{HLZ}]
\label{lem:center-critical}
Let $H$ be an edge-critical graph with $\chi(H)=r\geq2$, and let $H^*=H+u$.
If $v_1v_2$ is a critical edge of $H$, then both $uv_1$ and $uv_2$ are
critical edges of $H^*$.
\end{lemma}

\begin{lemma}[\cite{HLZ}]
\label{thm:HLZ-good}
Let $k,r\geq2$ be fixed integers, and let $\mathcal{H}_k=\{H_1,\ldots,H_k\}+v$, where every $H_i$ is edge-critical and satisfies $\chi(H_i)=r$. 
If $G$ is an $\mathcal{H}_k$-free graph on $n$ vertices and $\delta(G)\geq \frac{r-1}{r}n-k$,
then $G$ admits a $k$-good partition for sufficiently large $n$.
\end{lemma}

\section{The large-minimum-degree case}\label{s3}

To prove the upper bound in Theorem~\ref{thm:main}, we first consider edge-colorings for which every representing graph has large minimum
degree. In this setting, Lemma~\ref{thm:HLZ-good} provides a good partition, whose structural properties allow us to control the edges
inside and between its parts. Our main tool is an embedding lemma that extends critical edges to branches sharing a common
center. 

\begin{definition}\label{de:extension}
Let $G$ be a graph with a partition $V(G)=V_1\cup\cdots\cup V_r$, and let $z\in V_s$. 
\begin{enumerate}[label=\textup{(\alph*)}]
\item If there exist $x, y\in V_p$ for some $p\in[r]\setminus\{s\}$ such that $z\sim x, z\sim y,$ and $x\sim y$, then we call $xy$ a \emph{type-I seed edge}.

\item If there exists $w\in V_s$ such that $w\sim z$, then call $zw$ a \emph{type-I\!I seed edge}.
\end{enumerate}
\end{definition}

\begin{lemma}\label{lem:extension}
Let $t\geq 1$ and $r\geq 2$, and let $\mathcal H_t=\{H_1,\ldots,H_t\}+v$, where each $H_i$ is an edge-critical graph satisfying
$\chi(H_i)=r$. Put $h=\max\limits_{i\in[t]}|V(H_i+v)|$.
Let $G$ be a graph admitting a partition $V(G)=V_1\cup V_2\cup\cdots\cup V_r$.
Fix a vertex $z\in V_s$ for some $s\in[r]$ and a nonnegative integer $b$.
Assume that the following conditions hold:
\begin{enumerate}[label=\textup{(\roman*)}]

\item The spanning $r$-partite subgraph $G_{\rm cr}$ satisfies $e(G_{\rm cr})\ge \sum\limits_{1\le i<j\le r}|V_i||V_j|-b;$

\item There exist t seed edges in $G$ such that the edges of type-I and type-I\!I among them induce, respectively, a matching and a star graph; the matching and the star graph are vertex-disjoint;

\item $\min\limits_{j\in[r]}|V_j|>b+t(h-1)$.
\end{enumerate}
Then $G$ contains a copy of $\mathcal H_t$ with center $z$.
\end{lemma}
\begin{proof}
Suppose that $G$ is an $n$-vertex graph admitting a partition $V(G)=V_1\cup V_2\cup\cdots\cup V_r$ satisfying {\rm(i)}, {\rm(ii)} and {\rm(iii)}.
Let $e_1,\ldots,e_t$ be the $t$ seed edges given by condition (ii). Our aim is to embed $\mathcal H_t$ into $G$ to get the desired copy. To this end, we will successively employ the seed edge $e_i$ to construct a copy of $H_i+v$ centered at $z$ for every $i\in[t]$. 

Throughout the proof, when a vertex $u\in V(H_i+v)$ is embedded onto a vertex $x\in V(G)$, we call $x$ the \emph{image} of $u$. We prove our result through the following three steps.

\medskip
\noindent\textbf{Step 1.} Assigning the vertices of $H_i+v$ to the parts $V_1,\ldots,V_r$.
\vspace{3mm}

Fix $i\in[t]$. We assign the vertices of $H_i+v$ to the parts $V_1,\ldots,V_r$ according to the type of seed edge $e_i$.
Suppose that $e_i$ is of type-I. By Definition~\ref{de:extension}, we may write $e_i=x_iy_i\in E(G[V_{p_i}])$
for some $p_i\in[r]\setminus\{s\}$ such that $zx_i,zy_i\in E(G)$. Since $H_i$ is edge-critical, choose a critical edge $a_ib_i$ of $H_i$. So we have $\chi(H_i-a_ib_i)=r-1$.
Thus there exists a proper $(r-1)$-coloring of $H_i-a_ib_i$. Moreover, the vertices $a_i$ and $b_i$ must receive the same color.
Otherwise adding the edge $a_ib_i$ would still give a proper $(r-1)$-coloring of $H_i$.

Assign the color class containing $a_i$ and $b_i$ to $V_{p_i}$, and assign the remaining $r-2$ color classes
bijectively to the parts $\{V_j:j\in[r]\setminus\{s,p_i\}\}.$

Assign the center $v$ to $V_s$, and set
\[
a_i\mapsto x_i,\qquad b_i\mapsto y_i,\qquad v\mapsto z.
\]
Every edge of $H_i+v$ other than $a_ib_i$ then has its endpoints assigned to distinct parts. The critical edge $a_ib_i$ is mapped
to the seed edge $e_i=x_iy_i$, while the edges $va_i$ and $vb_i$ are realized by $zx_i$ and $zy_i$, respectively.

Now suppose that $e_i$ is of type-I\!I. By Definition~\ref{de:extension}, we may write $e_i=zw_i\in E(G[V_s])$.
Since $H_i$ is edge-critical, Lemma~\ref{lem:center-critical} implies that $H_i+v$ contains a critical edge incident with the center $v$, denoted by $va_i$. So we have $\chi((H_i+v)-va_i)=r.$

Fix a proper $r$-coloring of $(H_i+v)-va_i$. The vertices $v$ and $a_i$ must receive the same color. Otherwise, restoring
$va_i$ would give a proper $r$-coloring of $H_i+v$. Moreover, the color class containing $v$ and $a_i$ is exactly
$\{v,a_i\}$, since $v$ is adjacent to every vertex of $H_i$ other than $a_i$ in $(H_i+v)-va_i$.

Assign the color class $\{v,a_i\}$ to $V_s$, and assign the remaining $r-1$ color classes bijectively to the remaining
$r-1$ parts. Set
\[
v\mapsto z,\qquad a_i\mapsto w_i.
\]
Again, every edge of $H_i+v$ other than $va_i$ has its endpoints assigned to distinct parts, while the critical edge $va_i$ is
mapped to the seed edge $e_i=zw_i$.

Thus, in either case, the vertices of $H_i+v$ can be assigned to the parts $V_1,\ldots,V_r$. 
\vspace{3mm}

\medskip
\noindent\textbf{Step 2.}  Choose images for the vertices of $H_i+v$ within their assigned parts.
\vspace{3mm}

By Step 1, the images of the center $v$ and the endpoints of the chosen critical edge have already been fixed: the center $v$ is
mapped to $z$, while the endpoints of the critical edge are mapped to the endpoint(s) of the seed edge $e_i$.

We now choose images for all remaining vertices of $H_i+v$ one by one within their assigned parts. 
Let $u$ be the next vertex to be embedded, and suppose that $u$ is assigned to $V_j$. Note that $j\neq s$. If some neighbors of $u$ have already
been embedded, then the image of $u$ must be adjacent to their images. In particular, since $v$ has already been mapped to $z$ and $uv\in E(H_i+v)$, the image of $u$ must be adjacent to $z$.

Therefore, a vertex $x\in V_j$ cannot be chosen as the image of $u$ if $x$ is not adjacent to the image of some already embedded
neighbor of $u$. 
Since every neighbor of $u$ is assigned to a part different from $V_j$, this means that $x$ is incident with a missing
cross-edge. By condition (i), there are at most $b$ missing cross-edges in total. Since each such missing cross-edge has at most one endpoint in
$V_j$, at most $b$ vertices of $V_j$ are excluded by the adjacency relations of $H_i+v$.

Thus, when choosing an image for each vertex of $H_i+v$, at most $b$ vertices in the assigned part are excluded.
\vspace{3mm}

\medskip
\noindent\textbf{Step 3.} Embed the branches $H_1+v, H_2+v, \ldots, H_t+v$ successively.
\vspace{3mm}

We first embed the branch $H_1+v$. By Step 1, the center $v$ is mapped to $z$, and the chosen critical edge of $H_1+v$ is realized
by the seed edge $e_1$.

Consider a remaining vertex $u$ of $H_1+v$, and suppose that $u$ is assigned to $V_j$. By Step 2, at most $b$ vertices of $V_j$ are excluded. In addition, to ensure that distinct vertices of $H_1+v$ receive distinct images, we must avoid the images of the non-central vertices
of $H_1+v$ that have already been embedded. We also avoid all endpoints of the seed edges $e_2,\ldots,e_t$ so that these seed edges remain available
for embedding the subsequent branches. Therefore, there are at most
\[
\begin{aligned}
b+(h-2)+2(t-1)
&=b+t(h-1)-1-(t-1)(h-3)\\
&\leq b+t(h-1)-1
\end{aligned}
\]
unavailable vertices in $V_j$, where the inequality follows from $t\geq1$ and $h\geq3$. 

Hence, by condition~{\rm (iii)}, a suitable image for $u$ can always be chosen. Repeating this procedure, we can embed all the remaining vertices
of $H_1+v$ into their assigned parts. Since each vertex is chosen to be adjacent to all its previously embedded neighbors, together with the seed edge $e_1$ this gives a copy of $H_1+v$ in $G$.




We next consider a general branch $H_i+v$, where $2\leq i\leq t$. The branches $H_1+v,\ldots,H_{i-1}+v$ have already been embedded into $G$ such that their copies have the common center $z$, are pairwise vertex-disjoint outside $z$ and do not use any endpoints of the seed edges $e_i,\ldots,e_t$ other than $z$.

We apply the same embedding procedure as for $H_1+v$. At each step, at most $b$ vertices in the assigned part are
excluded. Moreover,  we must avoid the non-central vertices used by the
copies of $H_1+v,\ldots,H_{i-1}+v$, the images of non-central vertices of the current branch $H_i+v$ that have already been embedded, and all endpoints of the seed edges $e_{i+1},\ldots,e_t$.
Hence, there are at most
\[
\begin{aligned}
b+(i-1)(h-1)+(h-2)+2(t-i)
    &= b+(h-3)i+2t-1\\
    &\leq b+t(h-1)-1
    \qquad (\text{since } h\geq 3 \text{ and } i\leq t)
\end{aligned}
\]
unavailable vertices in the assigned part. Therefore, condition~{\rm (iii)} guarantees that the embedding of $H_i+v$ can be completed.

Repeating this procedure for every $i\in[t]$, we obtain copies of
$H_1+v,\ldots,H_t+v$ that have the common center $z$, are pairwise vertex-disjoint outside $z$. Moreover, for every $i\in[t]$, the seed edge $e_i$ represents a critical edge of $H_i+v$ in the constructed copy. Hence their union is the required copy of $\mathcal H_t$.
\end{proof}

A standard device in anti-Ramsey problems is a representing graph. Given an edge-coloring $c$ of $K_n$, a \textit{representing graph} is a spanning 
subgraph of $K_n$ obtained by choosing exactly one edge of each color. We denote the family of all representing graphs by $\mathcal{R}(c,K_n)$. Then
every graph in $\mathcal{R}(c,K_n)$ is rainbow and has exactly $|c(E(K_n))|$ edges.
We now apply the preceding embedding lemma to establish the upper bound under the large-minimum-degree condition.

\begin{lemma}\label{lem:high-degree}
Let $k\geq2, r\ge2$ be fixed integers, and let $\mathcal H_{k+1}=\{H_1,\ldots,H_{k+1}\}+v$, where each $H_i$ is an edge-critical graph satisfying $\chi(H_i)=r$. 
Let $c$ be an edge-coloring of $K_n$  with no rainbow copy of $\calH_{k+1}$ and  $n\geq g\bigl(r, |V(H_1)|, \ldots, |V(H_{k+1})|\bigr)$. If the minimum degree of every representing graph 
is at least 
$\frac{r-1}{r}n-(k+1)$, then $|c(E(K_n))|\leq\ex(n,\calH_{k}^{j})+1$ for every $j\in[k+1]$. 
\end{lemma}

\begin{proof}
Let $h=\max\limits_{i\in[k+1]}|V(H_i+v)|$. Fix $j\in[k+1]$, and assume for contradiction that $|c(E(K_n))| \geq \ex(n,\calH_{k}^{j})+2$. Take any $G\in\calR(c,K_n)$. As $G$ is rainbow and the coloring $c$ contains no rainbow copy of $\calH_{k+1}$, it follows that $G$ is $\calH_{k+1}$-free. 
By Lemma~\ref{thm:HLZ-good}, $G$ admits a $(k+1)$-good partition $V(G)=V_1\cup\cdots\cup V_r$. Recall that $G_{\rm cr}=G-(\bigcup\limits_{i=1}^r E(G[V_i]))$.

\begin{claim}\label{alba}
The $(k+1)$-good partition $V(G)=V_1\cup\cdots\cup V_r$ is almost balanced, in the sense that
$|V_i|=\frac nr+O(1)$ for every $i\in[r]$. Moreover, $e(G_{\rm cr})=\sum\limits_{1\leq i<j\leq r}|V_i||V_j|-O_{k,r}(1)$. 
\end{claim}
\begin{proof}[\bf Proof of Claim~\ref{alba}]
Let $G'$ be the minimal induced subgraph of $G$ such that $e(G')-\sum\limits_{1\le i<j\le r}|V'_i||V'_j|$ is maximal, where $V'_i=V(G')\cap V_i$ for each $i\in [r]$. By Lemma~\ref{maximal} (i), we have
\[
e(G)-\sum\limits_{1\le i<j\le r}|V_i||V_j|\leq e(G')-\sum\limits_{1\le i<j\le r}|V'_i||V'_j|\leq f(k,k).
\]
Combining this inequality with
$e(G)\ge \ex(n,\calH_{k}^{j})+2= t_r(n)+f(k-1,k-1)+2$, where the equality follows from Lemma~\ref{thm:HLZ-extremal}, we obtain
\begin{equation}\label{eq:P-close}
0\le t_r(n)-\sum_{1\le i<j\le r}|V_i||V_j|\le f(k,k)-f(k-1,k-1)-2.
\end{equation}

Write $|V_i|=n_i$. Observe that 
\[
\begin{aligned}
\sum\limits_{1\le i<j\le r}|V_i||V_j|&=\frac{1}{2}(n^2-\sum\limits_{i=1}^r n^2_i)
=\frac{1}{2}(n^2-\sum\limits_{i=1}^r (n_i-\frac{n}{r})^2-\frac{n^2}{r})
=\frac{r-1}{2r}n^2-\frac{1}{2}\sum\limits_{i=1}^r (n_i-\frac{n}{r})^2.
\end{aligned}
\]
Note that $\frac{r-1}{2r}n^2-\frac{r}{8}\leq t_r(n)\leq\frac{r-1}{2r}n^2$.
Using \eqref{eq:P-close}, we deduce
\[
\sum_{i=1}^r
\left(n_i-\frac nr\right)^2
\le
2(f(k,k)-f(k-1,k-1)-2)+\frac r4.
\]

Define $c_0=2\bigl(f(k,k)-f(k-1,k-1)-2\bigr)+\frac r4$. Then $\sum_{i=1}^r \left(n_i-\frac nr\right)^2 \le c_0$, so $(n_i-\frac nr)^2\le c_0$ for each $i\in[r]$.
Hence,
\[
\frac nr-\sqrt{c_0}
\le |V_i|
\le \frac nr+\sqrt{c_0}.
\]
Since $c_0$ depends only on $k$ and $r$, we conclude
\begin{equation}\label{eq:balanced-parts}
|V_i|=\frac nr+O_{k,r}(1)
\qquad\text{for every }i\in[r].
\end{equation}

We next estimate $e(G_{\rm cr})$. By Definition~\ref{def:good-partition}, condition~(i) gives $\Delta(G[V_i])\le k$; whereas condition (ii) holds for every index in $[r]$, applying it to any index different from $i$ yields $\nu(G[V_i])\leq k$. Consequently, $e(G[V_i]) \leq f(k, k)$ for every $i \in [r]$.
It follows that 
\begin{align}
e(G_{\rm cr})
&=e(G)-\sum_{i=1}^r e(G[V_i]) \notag\\
&\geq e(G)-rf(k,k) \notag\\
&\geq t_r(n)+f(k-1,k-1)+2-rf(k,k).
\label{ecr}
\end{align}
Combining~\eqref{ecr} with $e(G_{\rm cr})\leq \sum\limits_{1\le i<j\le r}|V_i||V_j|\leq t_r(n)$ yields
\[
\begin{aligned}
0
&\leq
\sum_{1\leq i<j\leq r}|V_i||V_j|-e(G_{\rm cr})
\leq t_r(n)-e(G_{\rm cr})
\leq rf(k,k)-f(k-1,k-1)-2.
\end{aligned}
\]
Therefore,
\[
    e(G_{\rm cr})
      =
      \sum_{1\leq i<j\leq r}|V_i||V_j|
      -O_{k,r}(1).
\]
Hence there exists a constant $d=d(k,r)$ such that $G$ has at most $d$ missing cross-edges between distinct parts.
\end{proof}

For each $i\in[r]$, define $T_i=\{v\in V_i: V(G)\setminus V_i\subseteq N_{G}(v)\}$. Then one sees that every vertex in $V_i\setminus T_i$ is incident to at least one missing cross-edge. Since there are at most $d$ missing cross-edges, the number of their endpoints is a constant independent of $n$. Combined with \eqref{eq:balanced-parts} gives us
\[
|T_i|=\frac nr+O(1)
\qquad\text{for every }i\in[r].
\]

\begin{claim}\label{norm}
For every $i\in[r]$, the coloring induced on $K_n[T_i]$ contains no
rainbow matching of size $k+1$.
\end{claim}
\begin{proof}[\bf Proof of Claim~\ref{norm}]
Fix $i\in[r]$. Suppose for contradiction that the edges in $\{e_q=x_qy_q\in E(K_n[T_i])$, $q\in[k+1]\}$ form a rainbow matching of size $k+1$. For each $q$, replace in $G$ the edge of color $c(e_q)$ by $e_q$. Since these colors are distinct, the resulting graph $G^*$ is again a representing graph. At most $k+1$ edges of $G$ are deleted.

We shall apply Lemma~\ref{lem:extension} to show that $G^*$ admits  a copy of $\mathcal H_{k+1}$. Note that $G^*$ has a partition $V_1,\ldots,V_r$. We are to verify that $G^*$ satisfies Lemma~\ref{lem:extension}(i), (ii) and (iii).

First, by Claim~\ref{alba}, $G$ has at most $d$ missing cross-edges. Since at most $k+1$ edges of $G$ are deleted during the construction of $G^*$, $G^*$ has at most $d+k+1$ missing cross-edges. Hence Lemma~\ref{lem:extension}(i) holds for $b=d+k+1$.

Next, choose an index $s\ne i$. By the definition of $T_i$, every vertex of $V_s$ is adjacent to every endpoint of $e_1,\ldots,e_{k+1}$ in $G$. A vertex of $V_s$ can lose this property in $G^*$ only if it is incident with a deleted edge in $E(G[V_i,V_s])$. Hence at most $k+1$ vertices of $V_s$ lose such property. By \eqref{eq:balanced-parts}, $|V_s|=\frac{n}{r}+O_{k,r}(1)>k+1$ for sufficiently large $n$, hence we can choose a vertex $z\in V_s$ adjacent in $G^*$ to both endpoints of every $e_q$ for $q\in[k+1]$. Moreover, since $e_q=x_qy_q\in E(G^*[V_i])$, and $z\in V_s$ ($s\ne i$) is adjacent to both $x_q$ and $y_q$, each $e_q$ is a type-I seed edge. Thus Lemma~\ref{lem:extension}(ii) holds for $t=k+1$.

Finally, note that $h,d,c_0,k$, and $r$ are fixed. By Claim~\ref{alba}, for sufficiently large $n$, we have
\[
    \min_{\ell\in[r]}|V_\ell|
    \geq \frac{n}{r}-\sqrt{c_0}
    >(d+k+1)+(k+1)(h-1).
\]
Thus Lemma~\ref{lem:extension}(iii) holds.

Therefore, $G^*$ admits a copy of $\mathcal H_{k+1}$ centered at $z$.
Because $G^*$ is a representing graph, this copy of $\mathcal H_{k+1}$ is rainbow, contradicting the assumption on $c$.
\end{proof}

\begin{claim}\label{monom}
There exists a constant $c_1=c_1(k,r)$ such that, for every $i\in[r]$, the colored complete graph $K_n[T_i]$ contains a
monochromatic matching $M_i$ satisfying $|M_i|\ge\frac{\frac nr-c_1-2k}{4k}$.
\end{claim}

\begin{proof}[\bf Proof of Claim~\ref{monom}]
By the estimate for $|T_i|$, one has
\begin{equation}\label{eq:Si-lower-bound}
|T_i|\ge\frac nr-c_1
\qquad\text{for every }i\in[r].
\end{equation}

Fix $i\in[r]$, and take a maximal rainbow matching $I$ in $K_n[T_i]$. By Claim~\ref{norm}, $|I|\le k$.
Let $U_i=T_i\setminus V(I)$.
Then
\begin{equation}\label{eq:Ui-lower-bound}
|U_i|
=
|T_i|-2|I|
\ge
|T_i|-2k.
\end{equation}

Every color appearing on an edge of $K_n[U_i]$ already appears on an edge of $I$. Otherwise an edge of a new color in
$K_n[U_i]$ could be added to $I$, contradicting the maximality of $I$. Consequently, the edges of $K_n[U_i]$ use at most $k$ distinct colors.

By the pigeonhole principle, there exists a color $\alpha_i$ that appears on at least
\[
\frac{1}{k}\binom{|U_i|}{2}
\]
edges of $K_n[U_i]$. 

Let $F_i$ be the spanning subgraph of $K_n[U_i]$ whose edge set consists of all edges colored $\alpha_i$, and let $M_i$ be a maximum matching in $F_i$. Then $V(M_i)$ is a vertex cover of $F_i$, which yields $e(F_i)\le2|M_i|(|U_i|-1)$.
On the other hand, $e(F_i)\ge\frac{|U_i|(|U_i|-1)}{2k}$.
For sufficiently large $n$, we have $|U_i|\ge2$. Comparing these two inequalities gives $|M_i|\ge\frac{|U_i|}{4k}$.
Combining \eqref{eq:Si-lower-bound} and \eqref{eq:Ui-lower-bound} gives us 
\[
\begin{aligned}
|M_i|
&\ge
\frac{|T_i|-2k}{4k}
&\ge
\frac{\frac nr-c_1-2k}{4k}.
\end{aligned}
\]
Hence $M_i$ is a monochromatic matching of the required size.
\end{proof}

Let $M_1\subseteq E(K_n[T_1])$ and $M_2\subseteq E(K_n[T_2])$ be monochromatic matchings of colors $\alpha_1$ and $\alpha_2$, respectively.
Deleting one edge of colors $\alpha_1$ and one edge of $\alpha_2$ when $\alpha_1 \neq \alpha_2$, and one edge of color $\alpha_1$ when $\alpha_1 = \alpha_2$ from $G$ yields a graph, say $G_0$. Note that the graph $G$ is a representing graph. Hence, it contains precisely one edge of color $\alpha_1$. Consequently, $G_0$ contains no edge of color $\alpha_1.$ Furthermore, 
\begin{equation}\label{eq:Gprime-lower}
e(G_0)=e(G)-|\{\alpha_1,\alpha_2\}|
 \ge\ex(n,\calH_{k}^{j})+2-|\{\alpha_1,\alpha_2\}|
 \ge\ex(n,\calH_{k}^{j}).
\end{equation}

\begin{claim}\label{clm:H-free}
The graph $G_0$ is $\calH_{k}^{j}$-free.
\end{claim}
\begin{proof}[\bf Proof of Claim~\ref{clm:H-free}]
Assume for contradiction that $G_0$ contains a copy, say $R,$ of $\calH_{k}^{j}$. Let $u\in V_s$ be the center of $R$. Note that $M_1$ and $M_2$ are monochromatic matchings of colors $\alpha_1$ and $\alpha_2$ with $V(M_1)\subseteq V_1$ and $V(M_2)\subseteq V_2$, respectively. Without loss of generality, we assume that $V_s\not= V_1.$ Note that $G_0$ is obtained from $G$ by deleting at most two edges. Denote by $L$ the set of all endpoints of the deleted edges. Thus, $|L|\leq4$.

By Claim~\ref{monom}, $M_1$ is a monochromatic matching of color $\alpha_1$ satisfying $|M_1|\ge\frac{\frac nr-c_1-2k}{4k}$.
Note also that $|V(R)|$ and $|L|$ are bounded independently of $n$. Hence, for
sufficiently large $n$, we can choose an edge $xy$ in $M_1$ such that $\{x,y\}\cap\bigl(V(R)\cup L\bigr)=\varnothing$.

Notice that $x,y\in T_1$ and $u\in V_s$ with $s\ne 1$. By the definition of $T_1$ one sees that $ux,uy\in E(G)$.
Moreover, $x, y\notin L$, and hence $ux,uy\in E(G_0)$. Bearing in mind that $G_0$ contains no edge of color $\alpha_1$ and $c(xy)=\alpha_1$. Hence, the graph $G_0+xy$ is rainbow. We shall apply Lemma~\ref{lem:extension} to show that $G_0+xy$ admits a copy of $H_j+v$ centered at $u$. 

By Claim~\ref{alba}, $G$ has at most $d$ missing cross-edges between distinct parts. Furthermore, it is known that $G_0$ is obtained from $G$ by deleting at most two edges. Consequently, the graph $G_0+xy$ has at most $d+2$ missing cross-edges. Hence $G_0+xy$ satisfies Lemma~\ref{lem:extension}(i) with $b=d+2$.

Furthermore, one sees that $xy$ is an edge of $(G_0+xy)[V_1]$, $u\in V_s$ with $s\ne 1$, and $ux,uy$ are edges in $G_0$. Hence, the edge $xy$ is a seed edge of type-I. Thus $G_0+xy$ satisfies Lemma~\ref{lem:extension}(ii) with $t=1$. 

Note that $c_0=2\bigl(f(k,k)-f(k-1,k-1)-2\bigr)+\frac{r}{4}$.
Since $h,d,k,r$ are fixed and $c_0$ depends only on $k$ and $r$, by Claim~\ref{alba}, one has, for sufficiently large $n$,
\begin{equation}\label{eq:avoidR}
\min_{\ell\in[r]}|V_\ell|
\geq \frac{n}{r}-\sqrt{c_0}
>(d+2)+(h-1)+|V(R)\setminus\{u\}|
>(d+2)+(h-1).
\end{equation}
Thereby, $G_0+xy$ satisfies Lemma~\ref{lem:extension}(iii). 

By Lemma~\ref{lem:extension}, one sees $G_0+xy$ admits a copy of $H_j+v$ centered at $u$. Note that $\{x,y\}\cap\bigl(V(R)\setminus\{u\}\bigr)=\emptyset$.
By \eqref{eq:avoidR}, the images of the vertices of $H_j$ can be chosen outside $V(R)\setminus\{u\}$. Hence this copy of
$H_j+v$ and $R$ are vertex-disjoint outside $u$. Therefore, their union forms a copy of $\mathcal H_{k+1}$ in $G_0+xy$. Since $G_0+xy$ is rainbow, this copy is rainbow, a contradiction, as desired.
\end{proof}

Together with Lemma~\ref{thm:HLZ-extremal}, ~\eqref{eq:Gprime-lower} and Claim~\ref{clm:H-free},  we have
\[
e(G_0)=\ex(n,\calH_{k}^{j})\qquad\text{and}\qquad G_0\in \mathcal G_{n,k,r}.
\]
Moreover, $\alpha_1\ne\alpha_2$. Hence $G_0$ is obtained from $G$ by deleting exactly two edges, and so $e(G)=e(G_0)+2=\operatorname{ex}(n,\mathcal H_k^j)+2$.
Note that $G\in\calR(c,K_n)$. Hence $|c(E(K_n))|=e(G)=\operatorname{ex}(n,\mathcal H_k^j)+2$.

Note that $G_0\in \mathcal G_{n,k,r}.$ Hence, $G_0$ is obtained from Tur\'an graph $T_r(n)$ by embedding a graph, say $J$, with $2k-1$ vertices, $k^2 - \frac{3k}{2}$ edges with maximum degree $k-1$ in one partite set if $k$ is even and embedding $2K_k$ in one partite set if $k$ is odd. For convenience, let $W_1\cup\cdots\cup W_r$ be the balanced partition of $V(T_r(n))$. Clearly, $W_1\cup\cdots\cup W_r$ is also a balanced partition of $V(G_0)$. 
The claim below describes the relationship between the new partition $W_1\cup\cdots\cup W_r$ of $V(G_0)$ and its original partition $V_1\cup \cdots \cup V_r$.

Recall that the \textit{symmetric group of degree} $r$, denoted by $S_r$, is defined as $S_r=\{\sigma \mid \sigma:[r]\to [r],\sigma \text{ is a bijection}\}$. Each element of $S_r$ is a \textit{permutation} on $\{1,2,\dots,r\}$. 
\begin{claim}\label{clm:partitions-and-J}
There exists a permutation $\sigma$ in $S_r$ such that $V_i=W_{\sigma(i)}$ for every $i\in[r]$. Moreover, the graph $2K_k$ (resp. $J$) contains a maximal matching of size $k-1$. 
\end{claim}

\begin{proof}[\bf Proof of Claim~\ref{clm:partitions-and-J}]
We begin by proving the first part of Claim \ref{clm:partitions-and-J}. 
As observed in the proof of Claim~\ref{alba}, we have $e(G[V_i])\leq f(k,k)$ for every $i\in[r]$. 
Therefore,
\begin{equation}\label{eq:Gprime-internal-bound}
e(G_0[V_i])
\le e(G[V_i])
\le f(k,k).
\end{equation}

For each $V_i$, choose $W_j$ such that $|V_i\cap W_j|=\max\limits_{p\in[r]}|V_i\cap W_p|$, and define $\sigma(i)=j$.
In this way, we obtain a map $\sigma:[r]\to[r]$. Next we are to show $\sigma$ is a permutation in $S_r.$ In fact, the sets $V_i\cap W_{1},\ldots,V_i\cap W_{r}$ form a partition of $V_i$, which yields $|V_i\cap W_{\sigma(i)}|\ge\frac{|V_i|}{r}$.
By Claim~\ref{alba}, $|V_i|\ge\frac nr-\sqrt{c_0}$.
It follows that
\[
|V_i\cap W_{\sigma(i)}|
\ge
\frac{n}{r^2}-\frac{\sqrt{c_0}}{r}.
\]

Suppose that $V_i\cap W_{q}\ne\emptyset$ for some $q\ne \sigma(i)$. Since $G_0$ contains the Tur\'an graph $T_r(n)$ with partite sets
$W_{1},\ldots,W_{r}$, we have 
\begin{equation}\label{e:9}
\begin{aligned}
e(G_0[V_i])
&\ge
|V_i\cap W_{q}|\,|V_i\cap W_{\sigma(i)}|
\ge
|V_i\cap W_{\sigma(i)}|
\ge
\frac{n}{r^2}-\frac{\sqrt{c_0}}{r}.
\end{aligned}
\end{equation}
It follows from \eqref{eq:Gprime-internal-bound} that $e(G_0[V_i])$ is bounded above by a finite constant, while \eqref{e:9} implies that $e(G_0[V_i])$ can be arbitrarily large for large $n$. This is impossible. 
Hence $V_i\cap W_{q}=\emptyset$ for all $q\ne \sigma(i)$, i.e.,  $V_i\subseteq W_{\sigma(i)}.$ Since
$i\in[r]$ is arbitrary, we obtain $V_i\subseteq W_{\sigma(i)}$ for every $i\in[r]$.
Note that both $V_1,\ldots,V_r$ and $W_{1},\ldots,W_{r}$ are partitions of the same vertex set $V(G_0)$ (each $V_i$ (resp. $W_j$) is nonempty). Hence
$\sigma(1),\ldots,\sigma(r)$ are pairwise distinct; otherwise, the sets $V_1\cup\ldots\cup V_r$ would be contained in $(W_1\cup\ldots \cup W_r)\setminus W_p$ for some $p\in[r]$, which implies that $W_p=\emptyset$, a contradiction.
Therefore, $\sigma\in S_r$. It follows that $V_i=W_{\sigma(i)}$ for every $i\in[r]$.

In what follows, we prove the second part of the claim.

For odd $k$, it is easy to see $K_{k}$ contains a matching of size $\frac{k-1}{2}$. 
Taking such a matching from each copy gives a matching in $2K_k$ of size $2\cdot\frac{k-1}{2}=k-1$.

We next consider the case when $k$ is even. When $k=2$, then $e(J)=2^2-\frac{3\times2}{2}=1$, and so $J$ contains a matching of size $1=k-1$.
Now let $k\ge4$, and suppose to the contrary that $\nu(J)\le k-2$. 

Since $\Delta(J)\le k-1$, Lemma~\ref{CH} gives
\begin{equation}\label{eq:10}
e(J)
\le f(k-2,k-1)
\le (k-2)(k-1)+(k-2)
=k^2-2k.
\end{equation}
On the other hand,
\(
e(J)
=k^2-\frac{3k}2
>k^2-2k,
\)
which contradicts \eqref{eq:10}. Hence $\nu(J)\ge k-1$.
Note that $J$ has $2k-1$ vertices. Consequently, we also have $\nu(J)\le k-1$. Therefore, $\nu(J)=k-1$. 
\end{proof}

By Claim~\ref{clm:partitions-and-J}, after relabeling $W_1,\ldots,W_r$, we may assume that
$V_i=W_i$ for every $i\in[r]$. Moreover, $J$ (or $2K_k$) admits a maximal matching $M$ of size $k-1$. 
Let $W_a$ denote the part in which $J$ (or $2K_k$) is embedded. Clearly, $V(M)\subseteq W_a$. 
Recall that $M_1$ and $M_2$ are monochromatic matchings of colors
$\alpha_1$ and $\alpha_2$ with $V(M_1)\subseteq W_1$ and $V(M_2)\subseteq W_2$, respectively. Without loss of generality, we may assume that $W_1\ne W_a$.

Note that $|M_2|\ge\frac{\frac nr-c_1-2k}{4k}$. Then, for sufficiently large $n$, we may choose an edge $e_2\in M_2$ so that $M\cup \{e_2\}$ is a matching.  
Choose any edge $e_1\in M_1$. By the construction of $G_0$, we have $G_0+e_1+e_2\in \calR(c,K_n)$. 
We next apply Lemma~~\ref{lem:extension} to show that $G_0+e_1+e_2$ contains a copy of $\mathcal H_{k+1}$ to derive a contradiction.

Since $G_0$ contains the $r$-partite Tur\'an graph with parts $W_1,\ldots,W_r$, so does $G_0+e_1+e_2$. Hence it satisfies Lemma~\ref{lem:extension}(i) for $b=0$.

For convenience, let $e_1=xy$. Note that the vertex $x\in W_1$, and $x$ is adjacent to both endpoints of every edge in $M\cup\{e_2\}$. Hence every edge in $M\cup\{e_2\}$ is a seed edge of type-I. The edge $e_1$ is of type-I\!I, and it is disjoint from $M\cup\{e_2\}$. Thus $G_0+e_1+e_2$ satisfies 
Lemma~\ref{lem:extension}(ii) for $t=k+1$.

Recall that $W_1,\ldots,W_r$ form a balanced partition of $V(G_0)$. For large $n$, one has
\[
    \min_{\ell\in[r]}|W_\ell|
    =\left\lfloor\frac{n}{r}\right\rfloor
    >(k+1)(h-1).
\]
Thus $G_0+e_1+e_2$ satisfies Lemma~\ref{lem:extension}(iii).

Therefore by Lemma~\ref{lem:extension}, $G_0+e_1+e_2$ admits a copy of $\mathcal H_{k+1}$ centered at $x$. As $G_0+e_1+e_2$ is rainbow, the copy of $\mathcal H_{k+1}$ is also rainbow, contradicting the assumption on $c$. This completes the proof.
\end{proof}

\section{Proof of Theorem~\ref{thm:main}}\label{s4}
In this section we prove Theorem~\ref{thm:main}. We begin by deriving the lower bound for $\ar(n,\calH_{k+1})$, then establish its upper bound. For ease of reading, we restate the theorem below:\vspace{2mm}

\noindent\textbf{Theorem~1.1.} 
{\it Let $r\ge2$ and $k\ge1$ be fixed, and let $H_1,\ldots,H_{k+1}$ be edge-critical graphs with $\chi(H_i)=r$ for every $i\in[k+1]$.
There is a function $\phi\bigl(r, |V(H_1)|,\ldots,|V(H_{k+1})|\bigr)$ such that if $n\ge\phi\bigl(r, |V(H_1)|,\ldots,|V(H_{k+1})|\bigr)$,
then
$
\ar(n,\calH_{k+1})=\ex(n,\calH_{k}^{j})+1
$
for every $j\in[k+1]$.}

\begin{proof}[\bf Proof of Theorem~\ref{thm:main}]

We first prove $\ar(n,\calH_{k+1})\ge \ex(n,\calH_{k}^{j})+1$ for every $j\in[k+1]$. 

Let $G\in\mathcal G_{n,k,r}$. By Lemma~\ref{thm:HLZ-extremal}, for every $j\in[k+1]$,
$G$ is $\mathcal H_k^j$-free and $e(G)=\ex(n,\calH_{k}^{j})$. 
Color every edge of $G$ with a distinct color, and color all edges of $K_n-E(G)$ in another color. That is to say, the total number of colors for $E(K_n)$ is 
$e(G)+1=\ex(n,\calH_{k}^{j})+1$.

Suppose that this edge-coloring of $K_n$ contains a rainbow copy of $\calH_{k+1}$. Since every edge outside $G$ receives the same color, the rainbow copy of $\calH_{k+1}$ 
can use at most one edge not in $G$. If such an edge is employed, remove the branch containing that edge. 
Otherwise, delete an arbitrary branch. 
The remaining $k$ branches form a copy of $\calH_{k}^{j}$ for some $j\in[k+1]$, and all of its edges lie in $G$. 
This contradicts the fact that $G$ is $\calH_{k}^{j}$-free.
Therefore, for every $j\in[k+1]$, we have $\ar(n,\calH_{k+1})\ge \ex(n,\calH_{k}^{j})+1$.

We now prove $\ar(n,\calH_{k+1})\le \ex(n,\calH_{k}^{j})+1$ for every $j\in[k+1]$. 

Fix $j\in[k+1]$. 
Suppose to the contrary that there is an edge-coloring $c$ of $K_n$ contains no rainbow copy of
$\calH_{k+1}$ with at least $\ex(n,\calH_{k}^{j})+2$ colors and $n\geq \phi\bigl(r,|V(H_1)|,\ldots,|V(H_{k+1})|\bigr)$, where $\phi\bigl(r,|V(H_1)|,\ldots,|V(H_{k+1})|\bigr)\gg g\bigl(r,|V(H_1)|,\ldots,|V(H_{k+1})|\bigr)$.

\textbf{Case~1: $k= 1$}.

By Lemma~\ref{thm:HLZ-extremal}, we have $\ex(n,\calH_1^{j})=t_r(n)$. Moreover, $\ex(n,\calH_2)=t_r(n)+f(1,1)=t_r(n)+1$. 
Let $G\in\mathcal R(c,K_n)$ be a representing graph. Then $e(G)=|c(E(K_n))|\ge t_r(n)+2$.
Thus $e(G)>\ex(n,\calH_2)$, and so $G$ contains a copy of $\calH_2$.
Since $G$ is rainbow, this copy of $\calH_2$ is rainbow in the
coloring $c$, a contradiction. Therefore $\ar(n,\calH_2)\le \ex(n,\calH_1^{j})+1$ for every $j\in[2]$.

\textbf{Case~2: $k\geq 2$}.

If every representing graph in $\mathcal R(c,K_n)$ has minimum degree at least $(r-1)n/r-(k+1)$,
then by Lemma~\ref{lem:high-degree}, the coloring $c$ uses at most $\ex(n,\calH_{k}^{j})+1$ colors, contradicting the assumption that $c$ uses at least $\ex(n,\calH_{k}^{j})+2$ colors. We may therefore suppose that there exists a representing graph $L^n\in \mathcal R(c,K_n)$ with $\delta(L^n)<\frac{r-1}{r}n-(k+1)$.
Thus there is a vertex $u_n\in V(K_n)$ satisfying $d_{L^n}(u_n)\le\lfloor (r-1)n/r\rfloor-(k+1)$.

Let $G^n=K_n$ and $ G^{n-1}=G^n-u_n$, and let $c_{n-1}$ be the coloring of $G^{n-1}$  inherited from $c$.
Since \(L^n\) contains exactly one edge of every color, deleting \(u_n\) can remove at most \(d_{L^n}(u_n)\) distinct colors from the coloring. Consequently, $G^{n-1}$ has at least
$
\ex(n,\calH_{k}^{j})+2-(\lfloor(r-1)n/r\rfloor-(k+1))
$
colors. Since $\ex(n,\calH_{k}^{j})-\ex(n-1,\calH_{k}^{j}) = t_r(n)-t_r(n-1) = \left\lfloor\frac{r-1}{r}n\right\rfloor$, the edge-colored complete graph $G^{n-1}$ admits at least $\ex(n-1,\calH_{k}^{j})+k+3$ distinct colors.

If every representing graph in $\mathcal R(c_{n-1},G^{n-1})$ has minimum degree at least $(r-1)(n-1)/r-(k+1)$,
then Lemma~\ref{lem:high-degree} implies that $c_{n-1}$ uses at most $\ex(n-1,\calH_{k}^{j})+1$ colors, which yields a contradiction.  Hence there exists a representing graph $L^{n-1}\in \mathcal R(c_{n-1},G^{n-1})$ and a vertex $u_{n-1}\in V(G^{n-1})$ such that $d_{L^{n-1}}(u_{n-1})\le\lfloor (r-1)(n-1)/r\rfloor-(k+1)$.

Repeating this argument, we may construct a sequence of edge-colored complete graphs
\[
G^n,G^{n-1},\ldots,G^{n-\ell}
\]
such that the number of colors of $G^{n-\ell}$ is at least $\ex(n-\ell,\calH_{k}^{j})+\ell(k+1)+2$, which is based on $\phi\bigl(r,|V(H_1)|, \ldots,|V(H_{k+1})|\bigr)\gg g\bigl(r,|V(H_1)|,\ldots,|V(H_{k+1})|)$. Since an edge-coloring of $G^{n-\ell}$ has at most $\binom{n-\ell}{2}$ colors, this will yield a contradiction for large $\ell$.

Therefore $\ar(n,\calH_{k+1})\leq \ex(n,\calH_{k}^{j})+1$ for every $j\in[k+1]$. 
By Lemmas~\ref{lem:AHS} and~\ref{thm:HLZ-extremal}, we obtain
\begin{equation*}
\operatorname{ar}(n,\mathcal{H}_{k+1})=\ex(n,\calH_{k}^{j})+1=t_r(n) +1+
\begin{cases}
k^2 - k & \text{if } k \text{ is odd},\\[4pt]
k^2 - \dfrac{3}{2}k & \text{if } k \text{ is even}.
\end{cases}
\end{equation*}
This completes the proof.
\end{proof}

\section{Concluding remarks}\label{s5}

Theorem~\ref{thm:main} determines the anti-Ramsey number of a suspension consisting of at least two edge-critical graphs. 
It is also natural to investigate the remaining case of the suspension formed by a single edge-critical graph.

We recall a classical result of Erd\H{o}s, Simonovits and S\'os~\cite{ESS}. 
For a graph $F$, define $\chi(F^{-})=\min\{\chi(F-e):e\in E(F)\}$.

\begin{theorem}[\cite{ESS}]
Let $F$ be a fixed graph with $\chi(F^{-})\geq 3$. Then $\operatorname{ar}(n,F)=t_{\chi(F^{-})-1}(n)+o(n^2)$.
\end{theorem}

In particular, let $F$ be an edge-critical graph with $\chi(F)=r\geq4$. Since $F$ contains a critical edge and
deleting a single edge decreases the chromatic number by at most one, we have $\chi(F^{-})=r-1\geq3$.
Consequently,
\[
    \operatorname{ar}(n,F)=t_{r-2}(n)+o(n^2).
\]

We now apply this result to the suspension $H+v$, where $H$ is edge-critical and $\chi(H)=r$. 
We have $\chi(H+v)=r+1$, and Lemma~\ref{lem:center-critical} implies that $H+v$ is edge-critical. Therefore,
for $r\geq3$,
\[
    \operatorname{ar}(n,H+v)=t_{r-1}(n)+o(n^2).
\]

Determining the exact value, however, appears to depend on the structural properties of $H$.
Complete graphs provide the most classical examples. If $H=K_2$, then $H+v=K_3$. Erd\H{o}s, Simonovits and S\'os~\cite{ESS}
showed that $\operatorname{ar}(n,K_3)=n-1$ for all $n\geq 3$.
On the other hand, if $H=K_r$ with $r\geq3$, then $H+v=K_{r+1}$, and Schiermeyer~\cite{Schiermeyer} proved that $\operatorname{ar}(n,K_{r+1})=\operatorname{ex}(n,K_r)+1=t_{r-1}(n)+1$ for every $n\geq r+1$.

These examples further demonstrate that the single edge-critical graph case cannot be recovered via a formal extension of
Theorem~\ref{thm:main}. In the setting of Theorem~\ref{thm:main}, at least two edge-critical graphs share a common center, and deleting
one of them leaves a non-trivial suspension whose Tur\'an number determines the exact anti-Ramsey number. When only one edge-critical
graph is present, deleting it leaves only the center, so there is no analogous sub-suspension. Consequently, the anti-Ramsey number in the one edge-critical graph case is determined by a different extremal structure.

This observation motivates the following open problem.

\begin{problem}
Let $H$ be an edge-critical graph with $\chi(H)=r\geq3$. Can the exact value of $\operatorname{ar}(n,H+v)$ be determined for
sufficiently large $n$? More specifically, does there exist a subgraph $H'$ of $H+v$ such that $\operatorname{ar}(n,H+v)$ can be
expressed in terms of $\operatorname{ex}(n,H')$?
\end{problem}

\section*{Disclosure statement}
The authors did not report any potential conflict of interest.
\section*{Acknowledgement}
S.L. financially supported by the National Natural Science Foundation of China (Grant Nos.  12571365, 12171190),  and the Open Research Fund of Hubei Key Laboratory of Mathematical Science (Grant No. MPL2026ORG006).
\section*{Data availability}
No data is available during the current study.

\end{document}